\documentclass{article}
\usepackage[utf8]{inputenc}

\usepackage{graphicx}
\usepackage{listings}
\usepackage[ruled,vlined]{algorithm2e}
\usepackage{amssymb}

\usepackage{amsmath}
\usepackage{amsthm}
\usepackage[english]{babel}

\newtheorem{thm}{Theorem}
\newtheorem{lem}{Lemma}

\newtheorem{defn}{Definition}

\begin{document}

\begin{center}
        \vspace*{1cm}
 
        \Large
        \textbf{On the Number of Representations for Simple Lie Groups}

    		\vspace{0.4cm}
    		\normalsize
    		\textbf{Mohammed Barhoush}
    
    		\vspace{0.9cm}
    		\textbf{Abstract}
\end{center}

The growth rate function $r_N$ counts the number of irreducible representations of simple complex Lie groups of dimension $N$. While no explicit formula is known for this function, previous works have found bounds for $R_N=\sum_{i=1}^Nr_i$. In this paper we improve on previous bounds and show that $R_N=O(N)$.

\section{Preliminaries}
This section provides a brief overview of the required preliminaries based from \cite{book}. A general background in Lie algebra and representation theory is assumed. 

The following elementary theory in Lie algebra is crucial for this paper.

\begin{thm}
For any complex simple Lie algebra $g$,

\begin{enumerate}
\item
Every finite dimensional representation $V$ of $g$ possess a highest weight vector.

\item
The subspace of V generated by the images of a highest weight vector $v$ under successive application of root spaces $g_i$ for some negative $i$ is an irreducible subrepresentation.

\item
An irreducible representation possess a unique highest weight vector up to scalar multiplication.
\end{enumerate}
\end{thm}

What this theorem entails is that an irreducible representation of a complex simple Lie algebra is determined by the highest weight. Furthermore, we know that a weight is always an integer combination of the functions $L_i$. This means every irreducible representation of complex simple Lie algebra is determined by a tuple of integers. We will detonate an irreducible representation by $V_\lambda$, where $\lambda$ is the tuple of integers. 

This paper will focus on the dimension of the irreducible representations of simple Lie groups. The simple Lie groups are $SO_n$, $Sl_n$ and $Sp_{2n}$ for some $n\in \mathbb{N}$. There are also five special cases $E_6$, $E_7$, $E_8$, $F_4$, and $G_2$.

\begin{thm}
If $\lambda=(\lambda_1,...,\lambda_n)$,then we can find $dim \ V_\lambda$.  
\\
\\
$SO_{2n}\mathbb{C}:$
$$dim(V_\lambda)=\prod_{1\leq i< j\leq n}\frac{l_i^2-l_j^2}{m_i^2-m_j^2}$$
Where $l_i=\lambda_i+n-i$ and $m_i=n-i$.

$SO_{2n+1}\mathbb{C}:$
$$dim(V_\lambda)=\prod_{1\leq i< j\leq n}\frac{l_i^2-l_j^2}{m_i^2-m_j^2}\prod_{1\leq i\leq n}\frac{l_i}{m_i}$$
Where $l_i=\lambda_i+n-i+\frac{1}{2}$ and $m_i=n-i+\frac{1}{2}$.

$Sp_{2n}\mathbb{C}:$
$$dim(V_\lambda)=\prod_{1\leq i< j\leq n}\frac{l_i^2-l_j^2}{m_i^2-m_j^2}\prod_{1\leq i\leq n}\frac{l_i}{m_i}$$
Where $l_i=\lambda_i+n-i+1$ and $m_i=n-i+1$.

$Sl_{n}\mathbb{C}:$
$$dim(V_\lambda)=\prod_{1\leq i< j\leq n}\frac{\lambda_i +...+\lambda_{j-1}}{j-i}$$
 
\end{thm}

Now we know how to calculate the dimension of a irreducible representation of a complex simple Lie group. This begs the question of how many of these representations have a certain dimension?

We define $r_N(G)$ to be the number of irreducible representations of a group $G$ with dimension $N$.

\begin{defn}
$$r_N=r_N(SO_{2n+1}\mathbb{C})+r_N(Sl_{n}(\mathbb{C})+r_N(SO_{2n}\mathbb{C})+r_N(Sp_{2n}\mathbb{C})$$
$$+r_N(E_6)+r_N(E_7)+r_N(E_8)+r_N(F_4)+r_N(G_2)$$
\end{defn}
Unfortunately there is no known explicit formula for this function. We can try to find bounds instead. However, $r_N$ fluctuates a lot so it is more useful to bound the sum function.

\begin{defn}
$$R_N(G)=\sum_{i=1}^{N}r_i(G)$$
$$R_N=\sum_{i=1}^{N}r_i$$
\end{defn}

A useful simple upper bound for this function is given by Theorem 1 in \cite{larsen}. 

\begin{thm}
If $G$ is  a simple compact Lie group and n is a positive integer,
$$R_N(G)\leq N$$
\end{thm}
The following result by \cite{larsen} is the best known upper bound for the rate growth function.
\begin{thm}
For any $\epsilon>0$ there exists $n\in \mathbb{N}$ such that if $G$ is a simple compact Lie group of dimension $\leq n$ then $r_N(G)\leq N^\epsilon$ for $N\geq 1$. 
\end{thm}

This gives an upper bound for $R_N$.
\begin{thm}
For all $\epsilon>0$, there exists $M \in \mathbb{N}$ such that $\forall N>M$
$$R_N<O(N^{1+ \epsilon})$$
\end{thm}
The problem is that as $\epsilon$ approaches $0$, $N$ approaches infinity.

In this paper, the following theorem is proved.
\begin{thm}
$$R_N<O(N)$$
\end{thm}

\section{$Sp_{2n}\mathbb{C}$, $SO_{2n}\mathbb{C}$ and $SO_{2n+1}\mathbb{C}$}

These three simple lie groups have very similar Weyl character formulas. That is why I will only prove that $R_N(SO(\mathbb{C})) \leq 13N$. The same steps can be used to prove that $R_N(SO(\mathbb{C})) \leq 13N$ and  $R_N(Sp(\mathbb{C})) \leq 13N$.

Instead of relating a highest weight vector with a tuple  $ \lambda =(\lambda_1,...,\lambda_n)$ where $\lambda_1\geq \lambda_2...\geq \lambda_n$, we will relate it to the tuple $l=(l_1,...,l_n)$ where $l_i=\lambda_i+n-i$. So every highest weight vector has a unique tuple $l$ such that $l_1>l_2>... \ l_n\geq 0$.

\begin{thm}
$\forall \ N \in \mathbb{N} , \ R_N(SO(\mathbb{C})) \leq 13N$.
\end{thm}
\begin{proof}~

I will prove this by strong induction on N. The base cases are satisfied since $R_1(SO\mathbb{C})=0\leq 13$ and $R_2(SO\mathbb{C})=1\leq 26$.

Assume that $\forall \ M <N, \ M \in \mathbb{N}, \ R_M(So_{2n}\mathbb{C}) \leq 13M$.

Now we need to prove that $R_N(SO(\mathbb{C})) \leq 13N$. I will do this by creating a map from tuples with dimensions less than or equal to $N$ to tuples with dimensions less than or equal to $\frac{9N}{16}$ and then use the inductive hypothesis.

For any tuple $l=(l_1,...,l_n)$ let
$$\prod(l_1,...,l_n)=\prod_{1\leq i< j\leq n}\frac{l_i^2-l_j^2}{m_i^2-m_j^2}=dim(V_l)$$

Consider a tuple $l$ that satisfies $\prod(l_1,...,l_n) \leq N$. There are three cases.
\begin{enumerate}
\item $l_1=l_2+1$

In this case we will remove the $l_1$ term from the tuple.
\begin{lem}
$\prod(l_2,...,l_n)\leq \frac{1}{4}\prod(l_1,...,l_n)=\frac{1}{4}N$
\end{lem}
\begin{proof}

$$\frac{\prod(l_1,...,l_n)}{\prod(l_2,...,l_n)}=\prod_{1<j\leq n}\frac{l_1^2-l_j^2}{m_1^2-m_j^2}$$

Notice that the above equation is minimized when $l_1$ is minimized and all the other $l_i$'s are maximized. Although there is a bigger decrease when minimizing $l_1$ than when maximizing the other $l_i$'s. But since we are not considering the case where $\lambda_i=0, \forall \ 1 \leq i \leq n$ and since $l_1>l_2>...l_n$ we get $l_1\geq n$. So all in all to minimize the product above we choose $l_i=i$.

$$\prod_{1<j\leq n}\frac{l_1^2-l_j^2}{m_1^2-m_j^2}>\prod_{1<i\leq n}\frac{{n}^2-i^2}{(n-1)^2-(i-1)^2}=\prod_{1<i\leq n}\frac{n+i}{n+i-2}$$
$$=\frac{(2n)(2n-1)}{(n)(n-1)}>4$$
$$\therefore \prod(l_2,...,l_n) < \frac{\prod(l_1,...,l_n)}{4}$$

\end{proof}
\item $l_1>l_2+1$ and $n>6$

We now consider $(\lfloor{\frac{l_1+l_2}{2}}\rfloor,l_2,...,l_n)$. Notice that since $l_1>l_2+1$ we get $\lfloor{\frac{l_1+l_2}{2}}\rfloor \geq \frac{l_2+2+l_2}{2} \geq l_2+1$.

\begin{lem}
 $\prod(\lfloor{\frac{l_1+l_2}{2}}\rfloor,l_2,...,l_n) \leq (3/4)^6*N$
\end{lem}
\begin{proof}
$$\prod(\lfloor{\frac{l_1+l_2}{2}}\rfloor,l_2,...,l_n) = \prod_{1<j\leq n}\frac{(\lfloor{\frac{l_1+l_2}{2}}\rfloor)^2-l_j^2}{m_1^2-m_j^2} \prod_{1< i<j \leq n}\frac{l_i^2-l_j^2}{m_i^2-m_j^2}$$

$$\leq \prod_{1<j\leq n}\frac{(\frac{l_1+l_2}{2})^2-l_j^2}{m_1^2-m_j^2} \prod_{1< i<j \leq n}\frac{l_i^2-l_j^2}{m_i^2-m_j^2}=\prod_{1<j\leq n}\frac{(\frac{l_1^2}{4}-\frac{l_1l_j}{2}-\frac{3l_j^2}{4})}{m_1^2-m_j^2} \prod_{1< i<j \leq n}\frac{l_i^2-l_j^2}{m_i^2-m_j^2}$$

$$\leq \prod_{1<j\leq n}\frac{(\frac{l_1^2}{4}-\frac{l_j^2}{2}-\frac{3l_j^2}{4})}{m_1^2-m_j^2} \prod_{1< i<j \leq n}\frac{l_i^2-l_j^2}{m_i^2-m_j^2} \leq \prod_{1<j\leq n}\frac{\frac{3}{4}(l_1^2-l_j^2)}{m_1^2-m_j^2} \prod_{1< i<j \leq n}\frac{l_i^2-l_j^2}{m_i^2-m_j^2}$$
$$\leq (\frac{3}{4})^6\prod(l_1,l_2,...,l_n)$$

The last inequality is because we consider only the case $n>6$.
\end{proof}

\item $n=2,3,4,5,6$

We will use the following result for this case.

\begin{thm}
If G is a simple compact Lie group and N is a positive integer,
$$R_N(G)\leq N$$
\end{thm}

Applying this theorem to $SO_{2n}\mathbb{C}$ we deduce that $R_N(SO_{2n}\mathbb{C})\leq N$ for $n=2,3,4,5,6$. These cases sum up to $5N$.

\end{enumerate} 

All in all, we have a map that takes a tuple $l$ of type 1 satisfying $\prod(l_1,l_2,...,l_n) \leq N$ and sends it to a tuple $m$ satisfying $\prod(m_1,m_2,...,m_n) \leq \frac{N}{4}$. This map also takes tuples of type 2 and sends them to tuples $b$ satisfying $\prod(b_1,b_2,...,b_n) \leq (\frac{3}{4})^6N$, where no more than two tuples are sent to the same tuple. Finally we add the tuples $l$ of type 3. Together this covers all tuples satisfying  $\prod(l_1,l_2,...,l_n) \leq N$. We can use this map along with the inductive hypothesis to deduce that:
$$R_N(SO(\mathbb{C})) \leq 2*R_{(\frac{3}{4})^6N}(SO(\mathbb{C}))+R_{N/4}(SO(\mathbb{C})) +5N$$
$$ \leq 2*13*(\frac{3}{4})^6N+13*\frac{N}{4}+5N \leq 13N$$
\end{proof}

\section{$Sl_{n}\mathbb{C}$}
\begin{thm}

$R_N(Sl(\mathbb{C}))<43N$
\end{thm}

\begin{proof}
First the function is split into 3 parts.
$$ R_N(Sl(\mathbb{C}))=\sum_{n=1}^{n=33}R_N(Sl_n(\mathbb{C})) +\sum_{n=34}^{n=\sqrt{N}}R_N(Sl_n(\mathbb{C})) +\sum_{n=\sqrt{N}+1}^{n=N-1}R_N(Sl_n(\mathbb{C})) $$

Each part is bounded separately. 
\begin{enumerate}
\item Using Theorem 1 it can deduce that:
$$\sum_{n=1}^{n=33}R_N(Sl_n(\mathbb{C})) <33N$$

\item Applying Theorem 4 with $\epsilon=\frac{1}{2}$ then,
$$R_N(Sl_n(\mathbb{C}))<\sqrt{N}$$
for $n>33$.

Now using this bound to all the terms in the sum. 
$$\sum_{n=34}^{n=\sqrt{N}}R_N(Sl_n(\mathbb{C}))<(\sqrt{N}-34)\sqrt{N}<N-34\sqrt{N}$$

\item I will bound this sum by bounding each term in the sum.

First consider the tuple $(4,1,1...,1)$ of length $\sqrt{N}$.
$$\prod(4,1,1....,1)=\frac{(3+\sqrt{N})(2+\sqrt{N})(1+\sqrt{N})}{6}>N$$
The last inequality is because $(3+\sqrt{N})>6$ since we are only considering $N>9$. Therefore the first term can only be 3 or 2.

Now consider the tuple $(1,3,1,1...1)$ of length $\sqrt{N}$.

$$\prod(1,3,1....,1)=\frac{(1+\sqrt{N})^2(\sqrt{N})(2+\sqrt{N})}{12}>N$$
Therefore the second term can only be 2.

Now consider the tuple $(1,1,2,...,1)$. 
$$\prod(1,1,2,...,1)>N$$
Therefore the third term can only take the value 1.

Notice that the size of $\prod(a_1,...,a_{\sqrt{N}})$ is increased more when you increase an $a_i$ closer to rather than further from the middle. This means that all terms other than the first two and last two can only take the value 1 in order to satisfy $\prod(a_1,...,a_n)\leq N$. This along with the constriction of what values the first two and last two terms can take, leaves very few options for tuples satisfying $\prod(a_1,...,a_{\sqrt{N}})\leq N$. With a little algebra you can cancel out some of these possibilities and you can show there is at most 7 different tuples satisfying  $\prod(a_1,...,a_{\sqrt{N}})\leq N$.

The same reasoning works on tuples of length $>\sqrt{N}$, therefore 
$$R_N(Sl_n\mathbb{C})<7$$
for $n>\sqrt{N}$.

We can now bound he sum.
$$\sum_{n=\sqrt{N}+1}^{n=N-1}R_N(Sl_n(\mathbb{C}))<7\sqrt{N}$$

\end{enumerate}

$$ R_N(Sl(\mathbb{C}))<(33N)+(N-34\sqrt{N})+(7\sqrt{N})=34N-27\sqrt{N}$$
\end{proof}

\section{Bound on $R_N$}
The bounds for $R_N(G)$ have been established for the four main simple Lie groups but there are still the five exceptions. These can be bounded using Theorem 1.
$$R_N(E_6)+R_n(E_7)+R_N(E_8)+R_N(F_4)+R_N(G_2)<5N$$

Combining all the linear bounds obtained in the paper,

$$R_N\leq 3*(13N)+34N-27\sqrt{N}+5N=O(N)$$

Notice that this bound is tight since $R_N(Sl)$ alone is bigger than $3N$.

\end{document}